\documentclass[12pt]{amsart}
\usepackage{hyperref}
\usepackage{amssymb,enumerate,csquotes,enumitem}
\usepackage[capitalise, nameinlink]{cleveref}
\usepackage[backend=bibtex,natbib=true]{biblatex}
\usepackage[all]{xy}
\makeatletter
\@namedef{subjclassname@2010}{%
  \textup{2010} Mathematics Subject Classification}
\makeatother

\newtheorem{theorem}{Theorem}[section]

\newtheorem{corollary}[theorem]{Corollary}
\newtheorem{lemma}[theorem]{Lemma}
\newtheorem{mainthm}{Theorem}

\theoremstyle{definition}
\newtheorem{definition}[theorem]{Definition}
\newtheorem{remark}[theorem]{Remark}
\newtheorem{example}[theorem]{Example}

\newcommand{\R}{\mathbb{R}}
\newcommand{\N}{\mathbb{N}}
\newcommand{\C}{\mathbb{C}}

\newcommand{\Z}{\mathbb{Z}}

\newcommand{\T}{\mathbb{T}}

\hypersetup{
colorlinks,
linkcolor=blue,          
filecolor=magenta,      
urlcolor=cyan           
}

\begin{document}

\title[Homotopical Stable Ranks]{Homotopical Stable Ranks for Certain $C^{\ast}$-algebras associated to Groups}
\author[Anshu Nirbhay, Prahlad Vaidyanathan]{Anshu Nirbhay, Prahlad Vaidyanathan}
\address{Department of Mathematics\\ Indian Institute of Science Education and Research Bhopal\\ Bhopal ByPass Road, Bhauri, Bhopal 462066\\ Madhya Pradesh. India.}
\email{ansh@iiserb.ac.in, prahlad@iiserb.ac.in}
\date{}
\begin{abstract}
	We study the general and connected stable ranks for $C^{\ast}$-algebras. We estimate these ranks for certain $C(X)$-algebras, and use that to do the same for certain group $C^{\ast}$-algebras. Furthermore, we also give estimates for the ranks of crossed product $C^{\ast}$-algebras by finite group actions with the Rokhlin property.
\end{abstract}
\maketitle

A noncommutative dimension (or rank) is a numerical invariant associated to a $C^{\ast}$-algebra that is meant to generalize the notion of Lebesgue covering dimension for topological spaces. First introduced by Rieffel in \cite{rieffel}, these ranks have grown to play an important role in $C^{\ast}$-algebra theory. In particular, algebras having low rank often enjoy regularity properties that are important in and of themselves, and in the context of the Elliott classification program.

Among the various notions of noncommutative dimension that now exist, we are interested in two such - the general stable rank (gsr), and the connected stable rank (csr). Introduced by Rieffel in his original paper, both these ranks are closely related to K-theory. As described below, these ranks not only control the behaviour of certain long exact sequences associated to K-theory, they are also homotopy invariant (and are hence collectively termed \emph{homotopical} stable ranks). This last property is crucial, and makes the study of these ranks different from those of other dimension theories.

The goal of this paper is to understand the behaviour of these ranks under certain natural constructions. We begin with group $C^{\ast}$-algebras associated to certain nilpotent groups. By a theorem of Packer and Raeburn \cite{packer}, these algebras may be expressed as the algebra of sections of a continuous field of $C^{\ast}$-algebras. Therefore, in order to estimate the ranks of these algebras, we are led to our first main theorem.

\begin{mainthm}\label{mainthm: fields}
Let $X$ be a compact metric space of finite covering dimension $N$, and let $A$ be a $C(X)$-algebra. Then
\[
csr(A) \leq \sup\{ csr(C(\T^N)\otimes A(x)) : x \in X\}.
\]
\end{mainthm}

Here, $A(x)$ denotes the fiber of $A$ at a point $x$ in $X$, and $\T^k$ denotes the $k$-fold product of the unit circle $S^1$. As mentioned above, \cref{mainthm: fields} leads to our next main result.

\begin{mainthm}\label{mainthm: groups}
Let $G$ be a discrete group that is a central extension $0 \to N\to G\to Q \to 0$, where $N$ is a finitely generated abelian group of rank $n$, and $Q$ is a free abelian group of rank $m$. Then
\[
csr(C^{\ast}(G)) \leq \biggl\lceil \frac{n+m}{2}\biggr\rceil + 1
\]
where $\lceil x\rceil$ denotes the least integer $\geq x$.
\end{mainthm}

We then turn to crossed product $C^{\ast}$-algebras by finite groups. We begin by giving an estimate for the connected stable rank of a crossed product $C^{\ast}$-algebra provided the underlying algebra has topological stable rank one (\cref{thm: csr_tsr_1}). However, in the case where the action has the Rokhlin property, we were able to obtain stronger estimates.

\begin{mainthm}\label{mainthm: rokhlin}
Let $\alpha : G\to \text{Aut}(A)$ be an action of a finite group $G$ on a separable, unital $C^{\ast}$-algebra $A$ with the Rokhlin property. Then
\[
csr(A\rtimes_{\alpha} G) \leq \biggl\lceil \frac{csr(A) - 1}{|G|}\biggr\rceil + 1
\]
and
\[
gsr(A\rtimes_{\alpha} G) \leq \biggl\lceil \frac{gsr(A) - 1}{|G|}\biggr\rceil + 1.
\]
In particular, if $csr(A) = 1$ or $gsr(A) = 1$, then the same is true for $A\rtimes_{\alpha} G$.
\end{mainthm}

Under the same hypotheses as above, Osaka and Phillips have shown in \cite{osaka} that if $A$ has either stable rank one or real rank zero, then the same is true for the crossed product $C^{\ast}$-algebra. Therefore, \cref{mainthm: rokhlin} may be thought of as more evidence that crossed products by Rokhlin actions preserves low dimension.

\section{Preliminaries} 

\subsection{Stable ranks}

Let $A$ be a unital $C^{\ast}$-algebra and $n$ be a natural number. A vector $\underline{a} := (a_1,a_2,\ldots, a_n) \in A^n$ is said to be \emph{left unimodular} if there exists a vector $(b_1,b_2,\ldots, b_n) \in A^n$ such that $\sum_{i=1}^n b_ia_i = 1$. We write $Lg_n(A)$ for the set of all left unimodular vectors. Note that $GL_n(A)$, the set of all invertible elements in $M_n(A)$, acts on $Lg_n(A)$ by left multiplication. 

\begin{definition}\label{defn: gsr}
Let $A$ be a unital $C^{\ast}$-algebra. The \emph{general stable rank (gsr)} of $A$ is the least integer $n \geq 1$ such that $GL_m(A)$ acts transitively on $Lg_m(A)$ for each $m\geq n$.
\end{definition}
If no such number $n$ exists, we simply write $gsr(A) = +\infty$. Furthermore, if $A$ is a non-unital $C^{\ast}$-algebra, then the general stable rank of $A$ is simply defined as that of $A^+$, the minimal unitization of $A$.  To avoid repetition, we adopt these same convention in the definitions of connected and topological stable ranks below.

\cref{defn: gsr} seems somewhat opaque, but it has a K-theoretic explanation. In what follows, we will assume $A$ is both unital, and has the
invariant basis number property \cite[Definition 1.36]{magurn}, so that we may make sense of the rank of certain modules over $A$. Now suppose $M$ is an $A$-module
such that $M\oplus A^s \cong A^{s+m}$ for integers $s, m > 0$, then we wish to know when we can conclude that $M\cong A^m$. Therefore, we consider the somewhat simpler situation of a finitely generated projective $A$-module $P$ together with an isomorphism
\[
f: P\oplus A\xrightarrow{\cong} A^n
\]
and we ask when $P\cong A^{n-1}$. Setting $Q := f(P\oplus \{0\})$ and $\underline{a} := f((0,1))$, we see that $Q\cong P$ and
\[
Q\oplus \underline{a}A = A^n
\]
It turns out that $\underline{a}$ is a left unimodular vector, and that $P\cong A^{n-1}$ if and only if there is an invertible $T \in GL_n(A)$ such that $T(\underline{a}) = e_n$, where $e_n = (0, 0, \ldots, 1) \in A^n$ \cite[Proposition 4.14]{magurn}. Hence, the general stable rank of $A$ determines the least rank at which a stably free projective module is forced to be free. 

Let $GL^0_n(A)$ denote the connected component of the identity element in $GL_n(A)$. Observe that $GL_n^0(A)$ is a normal subgroup of $GL_n(A)$, and hence acts on $Lg_n(A)$ as well.

\begin{definition}
Let $A$ be a unital $C^*$-algebra. Then the \emph{connected stable rank (csr)} of $A$ is the least integer $n \geq 1$ such that $GL_m^0(A)$ acts transitively on $Lg_m(A)$ for all $m\geq n$.
\end{definition}

This definition is, if possible, even more mysterious than the previous one. To understand its usefulness, we state a result due to Rieffel: Let $A$ be a unital $C^{\ast}$-algebra and $\theta_n : GL_n(A) \to GL_{n+1}(A)$ denote the natural inclusion
\[
a \mapsto \begin{pmatrix}
a & 0 \\
0 & 1
\end{pmatrix}.
\]
This induces a sequence of groups
\[
\ldots \to \frac{GL_{n-1}(A)}{GL_{n-1}^0(A)} \to \frac{GL_n(A)}{GL_n^0(A)} \to \ldots
\]
whose limit is $K_1(A)$.

\begin{theorem}\cite[Proposition 2.6 and Theorem 2.9]{rieffel2}\label{thm: k_1}
If $n\geq csr(A)$, then the map
\[
GL_{n-1}(A) \to K_1(A)
\]
is surjective. Furthermore, if $n\geq \max\{csr(A), gsr(C(\T)\otimes A)\}$, then
\[
\frac{GL_{n-1}(A)}{GL_{n-1}^0(A)} \cong K_1(A).
\]
\end{theorem}

Thus, these ranks together control the nonstable K-theory of a $C^{\ast}$-algebra. Before we proceed, we give one last definition, that of topological stable rank.

\begin{definition}
Let $A$ be a unital $C^{\ast}$-algebra. The \emph{topological stable rank (tsr)} of $A$ is the least integer $n \geq 1$ such that $Lg_n(A)$ is dense in $A^n$.
\end{definition}

As it turns out, if $Lg_n(A)$ is dense in $A^n$, then $Lg_m(A)$ is dense in $A^m$ for all $m\geq n$, which explains the difference between this and the earlier definitions.

\begin{remark}\label{rem: gsr_csr_properties}
We now list some basic properties of these ranks. While the original proofs are scattered through the literature, \cite{nica} is an immediate reference for all these facts.

\begin{enumerate}
\item $gsr(A\oplus B) = \max\{gsr(A),gsr(B)\}$. Analogous statements hold for $csr$ and $tsr$.
\item $gsr(A) \leq csr(A) \leq tsr(A) + 1$.
\item[] Strict inequalities are possible in both cases. In fact, it is possible that $tsr(A) = +\infty$, while $csr(A) < \infty$.
\item For any $n\in \N$,
\[
csr(M_n(A))\leq \biggl\lceil \frac{csr(A)-1}{n}\biggr\rceil + 1, \text{ and } gsr(M_n(A))\leq \biggl\lceil \frac{gsr(A)-1}{n}\biggr\rceil + 1.
\]
Here, $\lceil x\rceil$ refers to the least integer greater than or equal to $x$.
\item If $\pi : A\to B$ is a surjective $\ast$-homomorphism, then
\[
csr(B)\leq \max\{csr(A),tsr(A)\}, \text{ and } gsr(B)\leq \max\{gsr(A),tsr(A)\}.
\]
\item Furthermore, if $\pi :A\to B$ is a split epimorphism (ie. there is a $\ast$-homomorphism $s:B\to A$ such that $\pi\circ s = \text{id}_B$), then
\[
csr(B)\leq csr(A), \text{ and } gsr(B)\leq gsr(A).
\]
\item If $0\to J\to A\to B\to 0$ is an exact sequence of C*-algebras, then
\[
csr(A) \leq \max\{csr(J),csr(B)\}, \text{ and } gsr(A) \leq \max\{gsr(J),csr(B)\}.
\]
It is worth mentioning here that when $J$ is an ideal of $A$, then there is, a priori, no relation between the homotopical stable ranks of $A$ and those of $J$.
\item Let $\{A_i : i\in J\}$ be an inductive system of C*-algebras with $A := \lim A_i$. Then
\[
csr(A) \leq \liminf_i csr(A_i), \text{ and } gsr(A) \leq \liminf_i gsr(A_i).
\]
\item If $gsr(A) = 1$ (and hence if $csr(A) = 1$), then $A$ is stably finite. Conversely, if $gsr(A)\leq 2$ and $A$ is finite, then $gsr(A) = 1$.
\item If $csr(A) = 1$, then $K_1(A) = 0$. The converse is true if $tsr(A) = 1$.
\item If $tsr(A) = 1$, then $A$ has cancellation of projections, so $gsr(A) = 1$.
\item If $A$ and $B$ are homotopy equivalent (in the category of $C^{\ast}$-algebras), then $gsr(A) = gsr(B)$ and $csr(A) = csr(B)$.
\end{enumerate}
\end{remark}

\subsection{$C(X)$-algebras}

We now describe a class of $C^{\ast}$-algebras that we will focus on for the first part of the paper. From now on, $X$ will always be a compact Hausdorff space unless otherwise stated.

\begin{definition}\cite[Definition 1.5]{kasparov}
A unital $C^{\ast}$-algebra $A$ is said to be a \emph{$C(X)$-algebra} if there exists a unital $\ast$-homomorphism $\Phi \colon C(X) \longrightarrow Z(A)$, where $Z(A)$ denotes the centre of $A$.
\end{definition}

In other words, $A$ is a $C(X)$-module, so if $f\in C(X)$ and $a\in A$, we simply write $fa$ for $\Phi(f)(a)$. Let $Y \subset X$ be closed and let $C(X,Y)$ denote the set of all functions in $C(X)$ vanishing on $Y$. Then, $C(X,Y)A$ is an ideal in $A$ by the Cohen factorization theorem \cite[Theorem 4.6.4]{brown_ozawa}, so we write $A(Y) := A/C(X,Y)A$ for the corresponding quotient, and $\pi_Y : A\to A(Y)$ for the quotient map. Furthermore, if $Z\subset Y$ is another closed subset of $X$, then we write $\pi^Y_Z : A(Y)\to A(Z)$ for the natural quotient map satisfying $\pi_Z = \pi^Y_Z\circ \pi_Y$.

If $Y = \{x\}$ is a singleton set, then $A(x) := A(\{x\})$ is called the fiber of $A$ at $x$, and we write $\pi_x$ for the corresponding quotient map. For $a\in A$, we write $a(x)$ for $\pi_x(a)$.  For each $a\in A$, we have a map $\Gamma_a : X\to \R$ given by $x\mapsto \|a(x)\|$. This map is always upper semi-continuous \cite[Lemma 2.3]{kirchberg}. We say that $A$ is a \emph{continuous} $C(X)$-algebra if $\Gamma_a$ is continuous for each $a\in A$.

When $A$ is a $C(X)$-algebra, we will often have reason to consider other $C(X)$-algebras constructed from $A$. At that time, we will need the following remark.

\begin{remark}\label{rem: kirchberg}
Let $X$ be a compact, Hausdorff space and let $A$ be a $C(X)$-algebra. If $B$ is a nuclear $C^{\ast}$-algebra, then $A\otimes B$ carries a natural action of $C(X)$ given on elementary tensors by $f\cdot (a\otimes b) := (fa)\otimes b$. This makes $A\otimes B$ a $C(X)$-algebra, whose fiber at a point $x\in X$ is given by $A(x)\otimes B$.
\end{remark}

Finally, one fact that plays a crucial role for us is that a $C(X)$-algebra may be patched together from quotients in the following way: Let $B, C$ and $D$ be $C^{\ast}$-algebras and $\beta:  B\to D$ and $\gamma : C\to D$ be $\ast$-homomorphisms. The pullback of this system is defined to be
\[
A = B \oplus_{D} C = \{(b,c) \in B\oplus C:  \beta(b) = \gamma(c)\}.
\]
This is described by a diagram
\[
\xymatrix{
A \ar[d]_{\delta} \ar[r]^{\alpha} & B \ar[d]^{\beta}\\ 
C  \ar[r]_{\gamma} & D
}
\]
where $\alpha(b,c) = b$ and $\delta(b,c) = c$.

\begin{lemma}\cite[Lemma 2.4]{mdd_finite}\label{lem: pullback_cx_algebra}
Let $X$ be a compact, Hausdorff space and $Y$ and $Z$ be two closed subsets of $X$ such that $X = Y \cup Z$. If $A$ is a $C(X)$-algebra, then $A$ is isomorphic to the following pullback.
\[
\xymatrix{
A \ar[d]_{\pi_Z} \ar[r]^{\pi_Y} & A(Y) \ar[d]^{\pi_{Y \cap Z}^Y}\\ 
A(Z)  \ar[r]_{\pi_{Y \cap Z}^Z} & A(Y \cap Z)
}
\]
\end{lemma}

\subsection{Notational conventions} We fix some notation we will use repeatedly throughout the paper: We write $\T^k$ for the $k$-fold product of the circle $\T$. Given a $C^*$-algebra $A$ and a compact Hausdorff space $X$, we identify $C(X) \otimes A$ with $C(X,A)$, the space of continuous $A$-valued functions on $X$. If $X = \T^k$, we simply write $\T^kA$ for $C(\T^k,A)$. We write $\theta_n^A$ for the map $GL_{n-1}(A) \longrightarrow GL_n(A)$ given by 
\[
a\mapsto \begin{pmatrix}
a & 0 \\
0 & 1
\end{pmatrix}.
\]
If there is no ambiguity, we simply write $\theta^A$ for this map. Given a unital $\ast$-homomorphism $\varphi: A \to B$, we write $\varphi_n$ for the induced maps in a variety of situations, such as $M_n(A) \to M_n(B),GL_n(A) \to GL_n(B)$ and $Lg_n(A) \to Lg_n(B)$. Furthermore, when there is no ambiguity, we once again drop the subscript and denote the map by $\varphi$. 

Given two topological spaces $X$ and $Y$, we will write $[X,Y]$ for the set of free homotopy classes of continuous maps between them. If $X$ and $Y$ are pointed spaces, then we write $[X,Y]_{\ast}$ for the set of based homotopy classes of continuous functions based at those distinguished points. Here, we will be concerned with two pointed spaces associated to a unital $C^{\ast}$-algebra $A$: $GL_n(A)$, as a subspace of $M_n(A)$ with base point $I_n$; and $Lg_m(A)$, as a subspace of $A^m$ with base point $e_m$.

\section{$C(X)$-algebras}

The goal of this section is to prove \cref{mainthm: fields}. Given a $C(X)$-algebra, and a point $x\in X$, one often looks to use upper semi-continuity to propagate a given property from the fiber $A(x)$ to a neighbourhood of that point. Together with a compactness argument, one may then be able to ensure that the property holds for the entire $C(X)$-algebra. This is the basic approach to the theorem, the first step of which begins with the following lemma. 

\begin{lemma}\label{lem:open_orbit}
Let $A$ be a unital C*-algebra, $n\in \N$ and $e_n = (0,0,\ldots, 1_A) \in Lg_n(A)$. If $\underline{u} \in A^n$ is such that $\|\underline{u} - e_n\| < 1/n$, then there exists $S \in GL_n^0(A)$ such that $S\underline{u} = e_n$.
\end{lemma}
\begin{proof}
Consider
\[
T = \begin{pmatrix}
1 & 0 & 0 & \ldots & 0 & u_1 \\
0 & 1 & 0 & \ldots & 0 & u_2 \\
\vdots & \vdots & \vdots & \vdots & \vdots & \vdots \\
0 & 0 & 0 & \ldots & 1 & u_{n-1}\\
0 & 0 & 0 & \ldots & 0 & u_n
\end{pmatrix}
\]
then
\[
\|I_n - T\| \leq \sum_{i=1}^{n-1} \|u_i\| + \|1_A - u_n\| \leq n\|\underline{u}-e_n\| < 1.
\]
Hence, $T \in GL_n^0(A)$ and $T(e_n) = \underline{u}$, so $S := T^{-1}$ works.
\end{proof}

The proof of \cref{mainthm: fields} is by induction on the covering dimension of the underlying space $X$. The next result is the base case, and works even if the underlying space is not metrizable.

\begin{theorem}\label{thm:zero_dim_case}
Let $X$ be a zero dimensional compact Hausdorff space, and $A$ be a unital $C(X)$-algebra, then
\[
csr(A) \leq \sup\{csr(A(x)) : x\in X\}.
\]
\end{theorem}
\begin{proof} If $\sup\{csr(A(x)): x \in X\} = +\infty$, then there is nothing to prove, so we may assume that $\sup\{csr(A(x)): x \in X\} < \infty$. Fix $n\geq \sup\{csr(A(x)) : x\in X\}$, and $v\in Lg_n(A)$. We want to show that there exists $T\in GL_n^0(A)$ such that $Tv=e_n$.

For any $x\in X, v(x) \in Lg_n(A(x))$, so there exists $S\in GL_n^0(A(x))$ such that $Sv(x) = e_n(x)$. Since the quotient map $\pi_x : A\to A(x)$ is surjective, there exists $T_x \in GL_n^0(A)$ such that $T_x(x) = S$, then
\[
T_x(x)v(x) = e_n(x).
\]
By upper semi-continuity of the map $y \mapsto \|T_x(y)v(y) - e_n(y)\|$, there is a cl-open neighbourhood $U_x$ of $x$ such that, for each $y\in U_x$, we have
\[
\|T_x(y)v(y) - e_n(y)\| < 1/n.
\]
By \cite[Lemma 2.1, (ii)]{mdd_finite},
\[
\|\pi_{U_x}(T_xv) - \pi_{U_x}(e_n)\| < 1/n.
\]
Since $\pi_{U_x}$ is unital, $\pi_{U_x}(e_n) = e_n \in Lg_n(A(U_x))$. So, applying \cref{lem:open_orbit} to $\underline{u}:= \pi_{U_x}(T_xv)$, there exists $\widetilde{S_x} \in GL_n^0(A(U_x))$ such that
\[
\widetilde{S_x}\pi_{U_x}(T_xv) = \pi_{U_x}(e_n).
\]
Let $R_x := \widetilde{S_x}\pi_{U_x}(T_x) \in GL_n^0(A(U_x))$. Since $\dim(X) = 0$, we may choose a refinement of $\{U_x : x\in X\}$, the members of which are mutually disjoint. Since $X$ is compact, we obtain a finite subcover of that refinement, denoted by $\{V_1,V_2,\ldots, V_m\}$. Note that the $V_i$ are cl-open and mutually disjoint. By \cref{lem: pullback_cx_algebra},
\[
A \cong \bigoplus_{i=1}^m A(V_i)
\]
so $GL_n^0(A) \cong \bigoplus_{i=1}^m GL_n^0(A(V_i))$. For each $1\leq i\leq m$, there exists $R_i \in GL_n^0(A(V_i))$ such that
\[
R_i\pi_{V_i}(v) = \pi_{V_i}(e_n).
\]
Hence, there exists $T\in GL_n^0(A)$ such that $\pi_{V_i}(T) = R_i$ for all $1\leq i\leq m$, so that $Tv = e_n$. Thus, $GL_n^0(A)$ acts transitively on $Lg_n(A)$, whence $csr(A)\leq n$ as required.
\end{proof}

We now prove an analogous result for the general stable rank as well.

\begin{theorem}
Let $X$ be a zero dimensional compact Hausdorff space, and $A$ be a unital $C(X)$-algebra, then
\[
gsr(A) \leq \sup\{gsr(A(x)) : x\in X\}.
\]
\end{theorem}
\begin{proof}
The proof follows similar lines as the previous theorem: We assume that $\sup\{gsr(A(x)) : x\in X\} < \infty$, and fix $n\geq \sup\{gsr(A(x)): x\in X\}$ and $v\in Lg_n(A)$. We wish to construct $T \in GL_n(A)$ such that $Tv = e_n$. To that end, we fix $x\in X$, and see that there exists $S \in GL_n(A(x))$ such that $Sv(x) = e_n(x)$. Now choose $T_x \in M_n(A)$ such that $T_x(x) = S$ (Note that $T_x$ may not be invertible). Then
\[
T_x(x)v(x) = e_n(x).
\]
As before, there is a cl-open neighbourhood $U_x$ of $x$ such that
\[
\|\pi_{U_x}(T_xv) - \pi_{U_x}(e_n)\| < 1/n.
\]
Now since $S \in GL_n(A(x))$, there exists $\widetilde{T} \in M_n(A)$ such that $\widetilde{T}(x) = S^{-1}$. Hence,
\[
T_x(x)\widetilde{T}(x) = I_n(x).
\]
By \cref{rem: kirchberg}, $M_n(A)$ is a $C(X)$-algebra, so by upper semi-continuity, there is a cl-open neighbourhood $W_x$ of $x$ such that
\[
\|\pi_{W_x}(T_x)\pi_{W_x}(\widetilde{T}) - \pi_{W_x}(I_n)\| < 1.
\]
Hence, $\pi_{W_x}(T_x)$ is right-invertible. Similarly, there is a cl-open neighbourhood $\widetilde{W_x}$ of $x$ such that $\pi_{\widetilde{W_x}}(T_x)$ is left-invertible. Thus, replacing $U_x$ by $U_x\cap W_x\cap \widetilde{W_x}$, we may assume that $\pi_{U_x}(T_x) \in GL_n(A(U_x))$. The remainder of the argument goes through, \emph{mutatis mutandis}, from the previous theorem, to conclude that there exists $T \in GL_n(A)$ such that $Tv = e_n$. Thus, $GL_n(A)$ acts transitively on $Lg_n(A)$, whence $gsr(A) \leq n$ as required.
\end{proof}

Before we proceed, we record a fact that will be useful to us later.

\begin{remark}\label{rem: csr_increasing}
If $X$ is a compact Hausdorff space and $A$ is a unital $C^{\ast}$-algebra, then the evaluation map at a point of $X$ gives a split epimorphism $C(X,A)\to A$. By Property (5) of \cref{rem: gsr_csr_properties}, it follows that
\[
csr(A) \leq csr(C(X,A)).
\]
In particular, if $n\leq m$, then $csr(\T^nA) \leq csr(\T^m A)$.
\end{remark}

Now, we need the following definition from \cite{prahlad}: Let $D$ be a unital $C^{\ast}$-algebra, and consider the sequence of groups
\[
\{1_D\} = GL_0(D) \hookrightarrow GL_1(D)\hookrightarrow GL_2(D)\hookrightarrow \ldots
\]
where the maps are the natural inclusions $\theta_n^D$. For an integer $k\geq 0$, we get an induced sequence of groups (or sets if $k=0$)
\[
\pi_k(GL_0(D)) \to \pi_k(GL_1(D)) \to \pi_k(GL_2(D)) \to \ldots
\]
We write $\text{inj}_k(D)$ to be the least integer $n\geq 1$ such that the map 
\[
\pi_k(GL_{m-1}(D))\to \pi_k(GL_m(D))
\] is injective for each $m\geq n$. Similarly, we write $\text{surj}_k(D)$ to be the least integer $n\geq 1$ such that the map $\pi_k(GL_{m-1}(D))\to \pi_k(GL_m(D))$ is surjective for each $m\geq n$.

The next lemma is a strengthening of \cite[Proposition 2.7]{prahlad}. While not strictly speaking necessary for our arguments, this may help shed some light on the right-hand-side term appearing in \cref{mainthm: fields}.

\begin{lemma}\label{lem: csr_inj_surj}
Let $D$ be a unital $C^{\ast}$-algebra and $k\geq 1$ be an integer. Then
\[
\max\{\text{inj}_{k-1}(D), \text{surj}_k(D), csr(D)\} = csr(\T^k D).
\]
\end{lemma}
\begin{proof}
By \cite[Theorem 3.10]{prahlad}, it follows that
\[
csr(\T^k D) \leq \max\{\text{inj}_{k-1}(D), \text{surj}_k(D), csr(D)\}.
\]
To prove the reverse inequality, let $n\geq csr(\T^k D)$. Then by \cref{rem: csr_increasing},
\[
n \geq csr(D).
\]
Secondly, by Property (2) of \cref{rem: gsr_csr_properties}, we have $n \geq gsr(\T^k D)$. So by \cite[Theorem 3.7]{prahlad} applied to $X := \{\ast\}$ and $A := \T^{k-1}D$, we have
\begin{equation}\label{eqn:gsr_inequality}
n \geq \max\{gsr(\T^{k-1} D), \text{inj}_0(\T^{k-1}D)\}.
\end{equation}
Now by \cite[Corollary 4.2]{prahlad} and \cite[Lemma 4.3]{prahlad},
\[
gsr(\T^{k-1}D) \geq gsr(C(\Sigma \T^{k-2})\otimes D) \geq gsr(C(S^{k-1})\otimes D).
\]
Furthermore, by \cite[Lemma 3.8]{prahlad} and \cite[Proposition 2.7]{prahlad},
\[
gsr(C(S^{k-1})\otimes D)\geq \text{inj}_0(C(S^{k-1})\otimes D) = \text{inj}_{k-1}(D).
\]
So we conclude that
\[
n\geq \text{inj}_{k-1}(D).
\]
Finally, for $m\geq n$ fixed, we wish to show that the map $\pi_k(GL_{m-1}(D)) \to \pi_k(GL_m(D))$ is surjective. Since $m\geq csr(D)$, \cite[Corollary 8.5]{rieffel} implies that $Lg_m(D)$ is connected. Hence, by \cite[Corollary 1.6]{corach}, we have a long exact sequence
\[
\ldots \to \pi_{k+1}(Lg_m(D))\to \pi_k(GL_{m-1}(D))\to \pi_k(GL_m(D)) \to \pi_k(Lg_m(D)) \to \ldots
\]
Therefore, it suffices to show that $\pi_k(Lg_m(D)) = 0$.

Note that, since $m\geq csr(\T^k D)$, $Lg_m(\T^k D)$ is connected by \cite[Corollary 8.5]{rieffel}. Furthermore, the natural map $Lg_m(\T^k D) \to C(\T^k, Lg_m(D))$ given by evaluation is a homeomorphism by \cite[Lemma 2.3]{rieffel2}. Hence,
\[
[\T^k, Lg_m(D)] = \pi_0(Lg_m(\T^k D)) = 0
\] 
where the left hand side denotes the set of free homotopy classes of maps from $\T^k$ to $Lg_m(D)$. Now, by repeatedly applying \cref{eqn:gsr_inequality}, we see that $m\geq gsr(\T D)$. Since $m \geq csr(D)$, it follows by \cite[Lemma 2.6]{prahlad} that the forgetful map
\[
\pi_k(Lg_m(D)) \to [\T^k, Lg_m(D)]
\]
is bijective. Hence, $\pi_k(Lg_m(D)) = 0$ for all $m\geq n$ as required.
\end{proof}

The next lemma, which is crucial to our argument, is contained in the proof of \cite[Theorem 2.14]{prahlad}. We isolate it here in the form we need it.

\begin{lemma}\label{lem: pullback_csr}
Consider a pullback diagram of unital $C^{\ast}$-algebras
\[
\xymatrix{
	A\ar[r]^{\alpha}\ar[d]_{\beta} & B\ar[d]^{\delta} \\
	C\ar[r]^{\gamma} & D
}
\]
where either $\gamma$ or $\delta$ is surjective, and let $n$ be a natural number such that
\[
n \geq \max\{\text{inj}_0(D), \text{surj}_1(D)\}.
\]
Let $v\in Lg_n(A)$, and suppose that there exist $S_1\in GL_n^0(B)$ and $S_2 \in GL_n^0(C)$ such that
\[
S_1\alpha(v) = e_n \text{ and } S_2\beta(v) = e_n.
\]
Then, there exists $T \in GL_n^0(A)$ such that $Tv = e_n$.
\end{lemma}

\begin{remark}\label{rem:reorganize_cover}
As mentioned above, the proof of \cref{mainthm: fields} proceeds by induction on the covering dimension of the underlying space. What finally allows this argument to work is the following: If $X$ is a finite dimensional compact metric space, then covering dimension agrees with the small inductive dimension \cite[Theorem 1.7.7]{engelking}. Therefore, by \cite[Theorem 1.1.6]{engelking}, $X$ has an open cover $\mathcal{B}$ such that, for each $U \in \mathcal{B}$,
\[
\dim(\partial U)\leq \dim(X) - 1.
\]
Now suppose $\{U_1,U_2,\ldots, U_m\}$ is an open cover of $X$ such that $\dim(\partial U_i)\leq \dim(X)-1$ for $1\leq i\leq m$, we define sets $\{V_i : 1\leq i\leq m\}$ inductively by
\[
V_1 := \overline{U_1}, \text{ and } V_k := \overline{U_k\setminus \Bigl( \bigcup_{i<k} U_i\Bigr)} \text{ for } k>1
\]
and subsets $\{W_j : 1\leq j\leq m-1\}$ by
\[
W_j := \biggl(\bigcup_{i=1}^j V_i\biggr)\cap V_{j+1}.
\]
It is easy to see that $W_j \subset \bigcup_{i=1}^j \partial U_i$, so by \cite[Theorem 1.5.3]{engelking}, $\dim(W_j) \leq \dim(X)-1$ for all $1\leq j\leq m-1$.
\end{remark}

We are now in a position to prove \cref{mainthm: fields}. 

\begin{theorem}\label{thm: csr_cx_algebra}
Let $X$ be a compact metric space of finite covering dimension $N$, and let $A$ be a unital $C(X)$-algebra. Then
\[
csr(A) \leq \sup\{csr(\T^N A(x)) : x\in X\}.
\]
\end{theorem}
\begin{proof}
We assume by \autoref{thm:zero_dim_case} that $N\geq 1$ and the theorem is true for any $C(Y)$-algebra, where $Y$ is a compact metric space with $\dim(Y)\leq N-1$. So we assume that $\sup\{csr(\T^N A(x)) : x\in X\} < \infty$ and fix
\[
n\geq \sup\{csr(\T^N A(x)) : x\in X\}.
\]
We wish to show that $GL_n^0(A)$ acts transitively on $Lg_n(A)$. We begin, as before, with a vector $v\in Lg_n(A)$ and a point $x\in X$. By the first part of the proof of \cref{thm:zero_dim_case} (and using the fact that $X$ is locally compact), there is an open neighbourhood $U_x$ of $x$ and an operator $S_x \in GL_n^0(A)$ such that
\[
\pi_{\overline{U_x}}(S_xv-e_n) = 0.
\]
Furthermore, as in \cref{rem:reorganize_cover}, we may assume that $\dim(\partial U_x) \leq N-1$ for all $x\in X$. Now choose a subcover $\{U_1,U_2,\ldots U_m\}$ of $\{U_x : x\in X\}$, and define $\{V_i\}$ and $\{W_j\}$ as in \cref{rem:reorganize_cover}. Then each $V_i$ is a closed set and there are $S_i \in GL_n^0(A(V_i))$ such that
\[
\pi_{V_i}(S_iv - e_n) = 0
\]
for each $1\leq i\leq m$. We now induct on $m$ to produce an operator $T \in GL_n^0(A)$ such that $Tv = e_n$.

If $m=1$ there is nothing to prove, so suppose $m>1$, then $W_1 = V_1\cap V_2$ satisfies $\dim(W_1) \leq N-1$. By \cref{rem: kirchberg}, $\T A(W_1)$ is a $C(W_1)$-algebra with fibers $\{\T A(x) : x\in W_1\}$. So by induction hypothesis,
\begin{equation*}
\begin{split}
csr(\T A(W_1)) &\leq \sup\{csr(C(\T^{\dim(W_1)},\T A(x))) : x\in W_1\} \\
&= \sup\{csr(\T^{\dim(W_1)+1} A(x)) : x\in W_1\}.
\end{split}
\end{equation*}
But $\dim(W_1) \leq N-1$, so by \cref{rem: csr_increasing},
\[
csr(\T A(W_1)) \leq \sup\{csr(\T^N A(x)) : x\in W_1\} \leq n.
\]
By \cref{lem: csr_inj_surj}, 
\[
\max\{\text{inj}_0(A(W_1)), \text{surj}_1(A(W_1))\} \leq n.
\]
Now by \cref{lem: pullback_cx_algebra}, $A(V_1\cup V_2)$ is a pullback
\[
\xymatrix{
	A(V_1\cup V_2)\ar[r]\ar[d] & A(V_1)\ar[d] \\
	A(V_2)\ar[r] & A(W_1)
}
\]
Since the quotient maps in this diagram are surjective, \cref{lem: pullback_csr} allows us to construct $\widetilde{T} \in GL_n^0(A(V_1\cup V_2))$ such that
\[
\widetilde{T}\pi_{V_1\cup V_2}(v) = e_n.
\]

Now observe that $W_2=(V_1\cup V_2)\cap V_3$ and $\dim(W_2)\leq N-1$. Replacing $V_1$ by $V_1\cup V_2$ and $V_2$ by $V_3$ in the earlier argument, we may repeat the above procedure. By induction on $m$, we finally construct an element $T \in GL_n^0(A)$ such that $T(v) = e_n$. Thus, $GL_n^0(A)$ acts transitively on $Lg_n(A)$, so $csr(A) \leq n$ as required.
\end{proof}

\subsection{Application to Group $C^{\ast}$-algebras}

Let $G$ be a countable, discrete group which can be obtained as a central extension of the form
\[
0\to N\to G\to Q \to 0
\]
where $N$ and $Q$ are finitely generated, abelian groups and $Q$ is free (central means that the image of $N$ lies in the center of $G$). The goal of this section is to estimate the connected stable rank of $C^{\ast}(G)$.

To begin with, we briefly review the notion of a twisted group $C^{\ast}$-algebra in the discrete case (See \cite{packer2} for more details): Let $K$ be a discrete group. A multiplier (or \emph{normalized 2-cocycle} with values in $\T$) on $K$ is a map $\omega : K\times K \to \T$ satisfying
\[
\omega(s,1) = \omega(1,s) = 1 \text{ and } \omega(s,t)\omega(st,r) = \omega(s,tr)\omega(t,r)
\]
for all $s,t,r\in K$. Given a multiplier $\omega$ on $K$, we define an $\omega$-twisted convolution product and an $\omega$-twisted involution on $\ell^1(K)$ by
\begin{equation*}
\begin{split}
(f_1\ast f_2)(t) &:= \sum_{s \in K} f_1(s)f_2(s^{-1}t)\omega(s,s^{-1}t), \text{ and } \\
f^{\ast}(t) &:= \overline{\omega(t,t^{-1})f(t^{-1})}.
\end{split}
\end{equation*}
This makes $\ell^1(K)$ into a $\ast$-algebra, and its universal enveloping algebra is called the full twisted group $C^{\ast}$-algebra $C^{\ast}(K,\omega)$.

The following theorem of Packer and Raeburn allows us to use \cref{mainthm: fields} in this context. Note that, if $N$ is an Abelian group, $\widehat{N}$ denotes its Pontrjagin dual group.

\begin{theorem}\label{thm: packer_raeburn}\cite[Theorem 1.2]{packer}
Let $G$ be a countable, discrete, amenable group given as a central extension
\[
0 \to N \to G \to Q \to 0
\]
and let $\sigma$ be a multiplier on $G$ such that $\sigma(n,s) = \sigma(s,s^{-1}ns)$ for all $n\in N$ and $s\in G$. Then, $C^{\ast}(G, \sigma)$ is isomorphic to a continuous $C(\widehat{N})$-algebra whose fibers are twisted group C*-algebras of the form $C^{\ast}(Q,\omega)$.
\end{theorem}

We immediately conclude \cref{mainthm: groups}.

\begin{theorem}
Let $G$ be a discrete group that is a central extension $0 \to N\to G\to Q \to 0$ where $N$ is a finitely generated abelian group of rank $n$, and $Q$ is a free abelian group of rank $m$. Then
\[
csr(C^{\ast}(G)) \leq \biggl\lceil \frac{n+m}{2}\biggr\rceil+1.
\]
\end{theorem}
\begin{proof}
Note that $G$ is amenable, so by \cref{thm: packer_raeburn} (taking the trivial multiplier on $G$), $C^{\ast}(G)$ is a continuous $C(\widehat{N})$-algebra, each of whose fibers are of the form $C^{\ast}(Q,\omega)$ for some multiplier $\omega$ on $Q$. Now $\widehat{N} \cong \T^n \times F$ for some finite set $F$, so $\widehat{N}$ is a compact metric space of dimension $n$. Thus, by \cref{mainthm: fields},
\begin{equation}\label{eqn: csr_group}
csr(C^{\ast}(G)) \leq \sup_{\omega}\{csr(\T^n C^{\ast}(Q,\omega))\}.
\end{equation}
Now fix a multiplier $\omega$ on $Q$, and consider $B := C^{\ast}(Q,\omega)$. Define $Z_{\omega}$ to be the symmetrizer subgroup
\[
Z_{\omega} := \{x\in Q : \omega(x,y) = \omega(y,x) \quad\forall y\in Q\}
\]
and observe that, since $Q$ is abelian, $Z_{\omega}$ is central and satisfies the conditions of \cref{thm: packer_raeburn}. Therefore, $B$ is a $C(\widehat{Z_{\omega}})$-algebra. By \cref{rem: kirchberg}, $\T^n B$ is a $C(\widehat{Z_{\omega}})$-algebra, each of whose fibers are of the form $\T^n B(x)$. So by \cref{mainthm: fields},
\[
csr(\T^n B) \leq \sup\{ csr(\T^{N+n} B(x)) : x\in \widehat{Z_{\omega}}\}
\]
where $N = \dim(\widehat{Z_{\omega}})$. But $Z_{\omega}$ is a subgroup of $Q$, so has rank $\leq m$, whence $\dim(\widehat{Z_{\omega}}) \leq m$ (as above). Hence, by \cref{rem: csr_increasing},
\begin{equation}\label{eqn: csr_a_theta}
csr(\T^n B) \leq \sup\{ csr(\T^{m+n} B(x)) : x\in \widehat{Z_{\omega}}\}.
\end{equation}
Furthermore, by \cite[Theorem 1.5]{packer}, $\widehat{Z_{\omega}}$ is the primitive spectrum of $B$, so by the Dauns-Hoffmann theorem \cite[Theorem IV.1.6.7]{blackadar2}, each such fiber $B(x)$ is a simple $C^{\ast}$-algebra.

Now fix $x\in \widehat{Z_{\omega}}$ and note that by \cref{thm: packer_raeburn}, $B(x)$ is of the form $C^{\ast}(K,\sigma)$ where $K = Q/Z_{\omega}$ and $\sigma$ is a multiplier on $K$. Since $K$ is also a free Abelian group,
\[
C^{\ast}(K,\sigma) \cong A_{\theta}
\]
where $A_{\theta}$ is a simple non-commutative torus or $\C$. If $A_{\theta}$ is a simple non-commutative torus, then
\[
csr(\T^{\ell} A_{\theta}) \leq 2
\]
by \cite[Proposition 2.7]{rieffel2}. And if $A_{\theta} = \C$, then
\[
csr(\T^{\ell} A_{\theta}) = csr(C(\T^{\ell})) \leq \biggl\lceil \frac{\ell}{2}\biggr\rceil + 1
\]
by \cite[Corollary 2.5]{nistor}. Hence, for any point $x \in \widehat{Z_{\omega}}$, we see that
\[
csr(\T^{\ell} B(x)) \leq \biggl\lceil \frac{\ell}{2}\biggr\rceil + 1.
\]
Together with \cref{eqn: csr_group} and \cref{eqn: csr_a_theta}, this gives the required inequality.
\end{proof}

Note that we may as well have proved more. If $G$ is a central extension as above and $A = C^{\ast}(G)$, then by \cref{rem: kirchberg}, $\T A$ is a $C(\widehat{N})$-algebra, each of whose fibers are of the form $\T C^{\ast}(Q,\omega)$. Hence, the same argument shows that
\[
csr(\T A) \leq \biggl\lceil \frac{n+m+1}{2} \biggr\rceil + 1.
\]
Note that $csr(A) \leq csr(\T A)$ by \cref{rem: csr_increasing}, and $gsr(\T A)\leq csr(\T A)$ by Property (2) of \cref{rem: gsr_csr_properties}. Hence, by \cref{thm: k_1}, we see that
\[
K_1(A) \cong GL_k(A)/GL_k^0(A), \text{ where } k = \biggl\lceil \frac{n+m+1}{2} \biggr\rceil .
\]
As a simple example, consider $G$ to be the integer Heisenberg group \cite[Example 1.4 (1)]{packer}. Here, $G$ is a central extension
\[
0 \to \Z \to G \to \Z^2 \to 0
\]
so $n=1$ and $m=2$. Thus,
\[
K_1(A) \cong GL_2(A)/GL_2^0(A).
\]
\begin{remark}
To put our results in perspective, we consider the extreme cases of \cref{mainthm: groups}, namely when either $N$ or $Q$ is trivial. If $N$ is trivial, then $C^{\ast}(G) = C^{\ast}(Q) \cong C(\T^m)$. If $X$ is a compact Hausdorff space of dimension $m$, then
\[
csr(C(X)) \leq \biggl\lceil\frac{m}{2}\biggr\rceil + 1
\]
by a result of Nistor \cite[Corollary 2.5]{nistor}. Furthermore, Nica has shown that this upper bound is attained if the top cohomology group in $H^{odd}(X)$ is non-vanishing \cite[Theorem 5.3]{nica}. In particular, we conclude that
\[
csr(C^{\ast}(G)) = csr(C(\T^m)) = \biggl\lceil\frac{m}{2}\biggr\rceil + 1
\]
Now, if $Q$ is trivial, then $C^{\ast}(G) = C^{\ast}(N) \cong B\otimes C(\T^n)$ for some finite dimensional C*-algebra $B$. Since tensoring by a finite dimensional C*-algebra lowers the connected stable rank (by \cref{rem: gsr_csr_properties}), we conclude that
\[
csr(C^{\ast}(G)) = csr(B\otimes C(\T^n)) \leq csr(C(\T^n)) = \biggl\lceil\frac{n}{2}\biggr\rceil + 1
\]
\end{remark}
\section{Crossed Products by Finite groups}

Let $\alpha : G\to \text{Aut}(A)$ be an action of a finite group $G$ on a unital $C^{\ast}$-algebra $A$. The goal of this section is to estimate the homotopical stable ranks of the crossed product $A\rtimes_{\alpha} G$ in terms of the ranks of $A$.

A result of Jeong et al. \cite{jeong} states that the topological stable rank of $A\rtimes_{\alpha} G$ may be estimated by the formula
\[
tsr(A\rtimes_{\alpha} G) \leq tsr(A) + |G|-1.
\]
So by Property (3) of \cref{rem: gsr_csr_properties}, we conclude that
\[
csr(A\rtimes_{\alpha} G) \leq tsr(A) + |G|.
\]
Our first theorem is an improvement on this estimate in the case when $A$ has topological stable rank one, and builds on the ideas of \cite[Theorem 7.1]{rieffel}. Recall that $A\rtimes_{\alpha} G$ is generated by a copy of $A$ and unitaries $\{u_g : g\in G\}$ such that $u_gu_h = u_{gh}, u_{g^{-1}} = u_g^{\ast}$, and $u_gau_{g^{-1}} = \alpha_g(a)$ for all $g,h\in G$ and $a\in A$.

\begin{theorem}\label{thm: csr_tsr_1}
Let $G$ be a non-trivial finite group and let $\alpha : G\to \text{Aut}(A)$ be an action of $G$ on a unital $C^{\ast}$-algebra $A$. If $tsr(A) = 1$, then
\[
csr(A \rtimes_{\alpha}G) \leq |G|.
\]
\end{theorem}
\begin{proof}
We write $B := A\rtimes_{\alpha} G$, and enumerate $G$ as $\{g_0,g_1,\ldots, g_n\}$ where $g_0$ denotes the identity element of $G$. For each $b\in B$, there is a unique expansion
\[
b = \sum_{i=0}^n a_iu_{g_i}
\]
where $a_i \in A$. So we define the length of $b$ to be $L(b) := 1 + \max\{i \in \{0,1,\ldots, n\} : a_i \neq 0\}$, with the convention that $L(0) = 0$. For a vector $\underline{b} = (b_1,b_2,\ldots, b_m) \in B^m$, we write $L(\underline{b}) := \sum_{j=1}^m L(b_j)$.

Now fix $m\geq |G|$. We wish to show, as before, that $GL_m^0(B)$ acts transitively on $Lg_m(B)$. So fix $v\in Lg_m(B)$ and we wish to prove that there exists $S\in GL_m^0(B)$ such that $S(v) = e_1$ where $e_1 = (1_B,0,0,\ldots, 0)$. So consider
\[
V := GL_m^0(B)(v).
\]
Then $V$ is an open subset of $Lg_m(B)$ by \cite[Theorem 1]{corach2}, and hence of $B^m$. Let  $\underline{b} = (b_1,b_2,\ldots, b_m) \in V$ be a vector of minimal length in $V$. In other words, $s := L(\underline{b}) = \min\{L(\underline{z}) : \underline{z} \in V\}$. We claim that $b_i = 0$ for some $1\leq i\leq m$.

Suppose not, then choose $\epsilon > 0$ such that, for any vector $\underline{h} \in B^m$, $\|\underline{h} - \underline{b}\| < \epsilon$ implies that $\underline{h} \in V$, and write
\[
b_i = \sum_{j=0}^{t_i} a_{i,j} u_{g_j}
\]
where $t_i = L(b_i) -1$. By multiplying by permutation matrices if needed (note that this does not alter the value of $L(\underline{b})$), we may assume that
\begin{equation}\label{eqn: increasing_length}
0\leq t_1\leq t_2\leq \ldots \leq t_m.
\end{equation}
We now consider two cases:
\begin{enumerate}
\item Suppose first that there exists $1\leq i\leq m$ such that $t_i = t_{i+1}$. Once again multiplying by a permutation matrix, we may assume that $i=m-1$. (Note that the inequalities in \cref{eqn: increasing_length} need not hold after doing this). Now $tsr(A) = 1$, so there exists $x \in Lg_1(A)$ such that $\|x - a_{m-1,t_{m-1}}\| < \epsilon$. Consider
\[
\underline{h} = (b_1, b_2, \ldots, b_{m-2}, b_{m-1}', b_m)
\]
where $b_{m-1}' = \sum_{j=0}^{t_{m-1}-1} a_{m-1,j}u_{g_j} + xu_{g_{t_{m-1}}}$. Then, $\|\underline{h} - \underline{b}\| < \epsilon$, whence $\underline{h} \in V$. Furthermore, there exists $y\in A$ such that $yx = -a_{m,t_m}$, so consider
\[
\widetilde{T} := \begin{pmatrix}
1 & 0 & 0 & \ldots & 0 & 0 \\
0 & 1 & 0 & \ldots & 0 & 0 \\
\vdots & \vdots & \vdots & \vdots & \vdots & \vdots \\
0 & 0 & 0 & \ldots & 1 & 0 \\
0 & 0 & 0 & \ldots & y & 1
\end{pmatrix}
\]
then $\widetilde{T} \in GL_m^0(B)$ and $\widetilde{T}(\underline{h}) = \underline{h'} = (b_1,b_2,\ldots, b_{m-1}',b_m')$ where
\[
b_m' = \sum_{j=0}^{t_m-1} (ya_{m-1,j} + a_{m,j})u_{g_j}.
\]
This implies that $\underline{h'}\in V$ and $L(\underline{h'}) < L(\underline{b})$, contradicting the minimality of $L(\underline{b})$.
\item Now suppose there is no $1\leq i\leq m$ such that $t_i = t_{i+1}$, then by the Pigeon-hole principle, $m\leq |G|$, so it must happen that $m = |G|$. Since $G$ is non-trivial, $m\geq 2$, and it follows that
\[
b_1 = a_{1,0}u_{g_0} \text{ and } b_2 = a_{2,0}u_{g_0} + a_{2,1}u_{g_1}
\]
where $a_{1,0}\neq 0$ and $a_{2,1}\neq 0$. Since $tsr(A) = 1$, there exists $x\in A$ such that $\|x-a_{1,0}\| < \epsilon$ and $\alpha_{g_1}(x) \in Lg_1(A)$. Consider $\underline{h}=(b_1',b_2,b_3,\ldots, b_m)$ where
\[
b_1' = xu_{g_0}.
\]
Then $\|\underline{h} - \underline{b}\| < \epsilon$, so $\underline{h} \in V$. Furthermore, there exists $y\in A$ such that $y\alpha_{g_1}(x) = -a_{2,1}$, so if
\[
\widetilde{T} := \begin{pmatrix}
1 & 0 & 0 & \ldots & 0 & 0 \\
yu_{g_1} & 1 & 0 & \ldots & 0 & 0 \\
\vdots & \vdots & \vdots & \vdots & \vdots & \vdots \\
0 & 0 & 0 & \ldots & 1 & 0 \\
0 & 0 & 0 & \ldots & 0 & 1
\end{pmatrix}
\]
then $\widetilde{T} \in GL_m^0(B)$ and $\widetilde{T}(\underline{h}) = \underline{h'} = (b_1',b_2',b_3,\ldots, b_m)$ where
\[
b_2' = a_{2,0}u_{g_0}.
\]
Now observe that $\underline{h'} \in V$ and $L(\underline{h'}) < L(\underline{b})$, contradicting the minimality of $L(\underline{b})$.
\end{enumerate}
Hence, it follows that there is some $1\leq i\leq m$ such that $b_i = 0$. Then, multiplying by a permutation matrix once again, we assume that $b_1 = 0$. Set $\underline{b'} := (\epsilon, b_2,b_3,\ldots, b_m)$, then $\|\underline{b'} - \underline{b}\| < \epsilon$, so $\underline{b'}\in V$. Now set
\[
Q := \begin{pmatrix}
\frac{1}{\epsilon} & 0 & 0 & \ldots & 0 & 0 \\
\frac{-1}{\epsilon}b_2 & 1 & 0 & \ldots & 0 & 0 \\
\vdots & \vdots & \vdots  & \ldots & \vdots & \vdots \\
\frac{-1}{\epsilon}b_m & 0 & 0 & \ldots & 0 & 1
\end{pmatrix} \in GL_m^0(B).
\]
Then $Q(\underline{b'}) = e_1$, so $e_1 \in V = GL_m^0(B)v$. Thus, $GL_m^0(B)$ acts transitively on $Lg_m(B)$, whence $csr(A) \leq m$ as required.
\end{proof}

\subsection{Rokhlin Actions}

In this, the final section of the paper, our goal is to prove \cref{mainthm: rokhlin}. The following definition of the Rokhlin property is different from the original definition due to Izumi \cite[Definition 3.1]{izumi}, but the two are equivalent if the underlying algebra is separable (See \cite[Theorem 5.26]{phillips2}).

\begin{definition}\label{defn: rokhlin}
Let $\alpha : G\to \text{Aut}(A)$ be an action of a finite group $G$ on a unital, separable $C^{\ast}$-algebra $A$. We say that $\alpha$ has the \emph{Rokhlin property} if, for every finite set $F\subset A$ and every $\epsilon > 0$, there are mutually orthogonal projections $\{e_g : g\in G\} \subset A$ such that
\begin{enumerate}
\item $\alpha_g(e_h) = e_{gh}$ for all $g,h\in G$
\item $\|e_ga - ae_g\| < \epsilon$ for all $g\in G, a\in F$
\item $\sum_{g\in G} e_g = 1.$
\end{enumerate}
\end{definition}

The Rokhlin property may be thought of as a notion of freeness of the action, and has a number of interesting properties (See \cite[Chapter 13]{phillips}). In the context of noncommutative dimension, one result is known: Osaka and Phillips have shown in \cite{osaka} that a variety of different classes of $C^{\ast}$-algebras are closed under crossed products by finite group actions with the Rokhlin property. In particular, if $A$ has topological stable rank one or real rank zero, then so does $A\rtimes_{\alpha} G$. We prove an analogous result for homotopical stable ranks. In fact, we estimate these ranks for the crossed product $C^{\ast}$-algebra in terms of those of $A$.

In what follows, if $\alpha : G\to \text{Aut}(A)$ is an action of a group $G$ on a C*-algebra $A$, we write $A^{\alpha}$ to denote the fixed point subalgebra of $A$. To begin with, we need a result due to Gardella \cite{gardella_rokhlin}. We are grateful to the referee for pointing out this result (and the subsequent line of reasoning) to us, as it considerably simplified our original argument.

\begin{theorem}\cite[Theorem 2.11]{gardella_rokhlin}\label{lem: gardella}
Let $\alpha : G\to \text{Aut}(A)$ be an action of a finite group $G$ on unital, separable C*-algebra $A$ with the Rokhlin property. Then, there is a sequence of unital,  completely positive, contractive linear maps $\psi_n : A\to A^{\alpha}$ such that, for all $a,b\in A$, we have
\[
\lim_{n\to \infty}\|\psi_n(ab) - \psi_n(a)\psi_n(b)\| = 0
\]
and, for all $a\in A^{\alpha}$, we have $\lim_{n\to \infty}\|\psi_n(a) - a\| = 0$.
\end{theorem}

\begin{lemma}\label{lem: fixed_point}
Let $\alpha : G\to \text{Aut}(A)$ be an action of a finite group $G$ on a unital, separable C*-algebra with the Rokhlin property. Then
\[
gsr(A^{\alpha}) \leq gsr(A) \text{ and } csr(A^{\alpha}) \leq csr(A).
\]
\end{lemma}
\begin{proof}
Let $\psi_j : A\to A^{\alpha}$ be a sequence of approximately multiplicative maps as in \cref{lem: gardella}. For each $n \in \N$, let $\psi_j^{(n)} : M_n(A)\to M_n(A^{\alpha})$ denote its inflation.

We begin with the $gsr$ inequality. We assume that $gsr(A) < \infty$ and let $n \geq gsr(A)$, and fix $\underline{v} \in Lg_n(A^{\alpha})$, and $\epsilon > 0$. By hypothesis, there exists $T \in GL_n(A)$ such that $T\underline{v} = e_n$. Now choose $\eta > 0$ so that if $S \in M_n(A^{\alpha})$ such that $\|S - T\| < \eta$, then $S \in GL_n(A^{\alpha})$. Furthermore, we may assume that $\eta\|\underline{v}\| < \epsilon$.

Then, since $\|\psi_j^{(n)}(T) - T\| \to 0$ as $j\to \infty$, we may choose $j\in \N$ such that, if $S := \psi_j^{(n)}(T)$, then $\|S - T\| < \eta$. This implies that $S \in GL_n(A^{\alpha})$ and that
\[
\|S\underline{v} - e_n\| = \|S\underline{v} - T\underline{v}\| < \eta\|\underline{v}\| < \epsilon.
\]
This is true for every $\epsilon > 0$. Since the action of $GL_n(A^{\alpha})$ on $Lg_n(A^{\alpha})$ has closed orbits by \cite[Theorem 1]{corach2}, we conclude that $GL_n(A^{\alpha})$ acts transitively on $Lg_n(A^{\alpha})$ as required.

The proof of the $csr$ inequality is similar, except we need the fact that the $GL_n^0(A^{\alpha})$ orbits are also closed, by \cite[Theorem 8.3]{rieffel}.
\end{proof}

The next lemma is a simpler version of a result due to Gardella, and we state it in the form that we need.

\begin{lemma}\cite[Theorem 3.8]{gardella_compact}\label{lem: tensor_rokhlin}
Let $\alpha : G\to \text{Aut}(A)$ be an action of a finite group on a separable, unital $C^{\ast}$-algebra $A$, and let $\beta : G\to \text{Aut}(B)$ be an action of $G$ on a unital, nuclear C*-algebra $B$. Let $\alpha\otimes \beta : G\to \text{Aut}(A\otimes B)$ denote the tensor product action $(\alpha\otimes \beta)_g = \alpha_g\otimes \beta_g$. If $\alpha$ has the Rokhlin property, then so does $\alpha\otimes \beta$.
\end{lemma}

We are now in a position to prove \cref{mainthm: rokhlin}.

\begin{theorem}
Let $\alpha : G\to \text{Aut}(A)$ be an action of a finite group $G$ on a separable, unital $C^{\ast}$-algebra $A$ with the Rokhlin property. Then
\[
csr(A\rtimes_{\alpha} G) \leq \biggl\lceil \frac{csr(A) - 1}{|G|}\biggr\rceil + 1
\]
and
\[
gsr(A\rtimes_{\alpha} G) \leq \biggl\lceil \frac{gsr(A) - 1}{|G|}\biggr\rceil + 1.
\]
In particular, if $csr(A) = 1$ or $gsr(A) = 1$, then the same is true for $A\rtimes_{\alpha} G$.
\end{theorem}
\begin{proof}
We begin with the $gsr$ inequality: Let $B := \mathcal{B}(\ell^2(G))$, and let $\text{Ad}(\lambda) : G\to \text{Aut}(B)$ denote the natural action induced by the left regular representation of $G$ on $\ell^2(G)$. Let $\alpha\otimes \text{Ad}(\lambda) : G\to \text{Aut}(A\otimes B)$ denote the tensor product action, which has the Rokhlin property by \cref{lem: tensor_rokhlin}. However, it follows from noncommutative duality that
\[
A\rtimes_{\alpha} G \cong (A\otimes B)^{\alpha\otimes \text{Ad}(\lambda)}.
\]
Hence, by \cref{lem: fixed_point}
\[
gsr(A\rtimes_{\alpha} G) \leq gsr(A\otimes B) = gsr(M_{|G|}(A)) \leq \biggl\lceil \frac{gsr(A) - 1}{|G|}\biggr\rceil +1
\]
where the last inequality follows from \cite[Corollary 11.5.13]{mcconnell} (Note that the rank $glr(A)$ used in \cite{mcconnell} is the same as $gsr(A) - 1$).

The $csr$ inequality is entirely similar, except at the very end, we need the fact that
\[
csr(M_{|G|}(A)) \leq \biggl\lceil \frac{csr(A) - 1}{|G|}\biggr\rceil + 1
\]
which was proved by Rieffel \cite[Theorem 4.7]{rieffel2}.
\end{proof}

We end with some examples that illustrate these results. 

\begin{corollary}
Let $\alpha : G\to \text{Aut}(A)$ be an action of a finite group $G$ on a unital, separable C*-algebra $A$ with the Rokhlin property. Then
\[
csr(\T B) \leq \biggl\lceil \frac{csr(\T A) - 1}{|G|}\biggr\rceil + 1
\]
and
\[
gsr(\T B) \leq \biggl\lceil \frac{gsr(\T A) - 1}{|G|}\biggr\rceil + 1.
\]
Hence, if
\[
n \geq \max\biggl\lbrace \biggl\lceil \frac{gsr(\T A)-1}{|G|}\biggr\rceil, \biggl\lceil\frac{csr(A)-1}{|G|} \biggr\rceil\biggr\rbrace
\]
then the natural map
\[
GL_n(B)/GL_n^0(B) \to K_1(B)
\]
is an isomorphism.
\end{corollary}
\begin{proof}
Let $\beta : G\to \text{Aut}(C(\T))$ be the trivial action. Then, by \cref{lem: tensor_rokhlin}, the action $\alpha\otimes \beta : G\to \text{Aut}(\T A)$ has the Rokhlin property, so the result follows from \cref{mainthm: rokhlin}. The final conclusion now follows from \cref{thm: k_1}.
\end{proof}

\begin{example}
Let $G$ be a finite group of order $k$ and $A$ be a UHF algebra of type $k^{\infty}$. Then $A$ admits a Rokhlin action $\alpha : G\to \text{Aut}(A)$ by \cite[Example 3.2]{izumi}, so the above corollary applies to the algebra $B := A\rtimes_{\alpha} G$. Since $A$ is divisible in the sense of \cite[Definition 4.1]{rieffel2}, and $tsr(A) = 1$, it follows from \cite[Corollary 4.12]{rieffel2} that $csr(\T A)\leq 2$. So by the above corollary, we conclude that the map
\[
GL_1(B)/GL_1^0(B) \to K_1(B)
\]
is an isomorphism. In fact, the same argument shows that 
\[
csr(\T^k B) \leq 2
\]
for all $k \in \N$. Hence, by \cite[Theorem 3.3]{rieffel2}, the natural inclusion map $\theta^B : GL_n(B)\to GL_{n+1}(B)$ induces a weak homotopy equivalence for all $n\geq 1$.
\end{example}
We end with an example that shows that \cref{mainthm: rokhlin} does not hold if the action does not satisfy the Rokhlin property.

\begin{example}
If $A = M_{2^{\infty}}$ denotes the UHF algebra of type $2^{\infty}$, then
\[
csr(A) = 1.
\]
If $G = \Z/2\Z$, then Blackadar has constructed in \cite{blackadar_car} an action of $G$ on $A$ such that $K_1(A\rtimes G) \neq \{0\}$. It follows by Property (9) of \cref{rem: gsr_csr_properties} that
\[
csr(A\rtimes G) > 1.
\]
Thus, \cref{mainthm: rokhlin} does not hold in this situation.

Furthermore, $tsr(A) = 1$, so by \cref{thm: csr_tsr_1}, we conclude that
\[
csr(A\rtimes G) = 2.
\]
Thus, this example also shows that the estimate in \cref{thm: csr_tsr_1} is sharp.
\end{example}

\subsection*{Acknowledgements}
The second named author was partially supported by SERB Grant YSS/2015/001060. The authors would like to thank the referee for a careful reading of the manuscript, and for suggesting changes that significantly improved the exposition.

\printbibliography
\end{document}